\documentclass[numbook,envcountsame,smallcondensed,draft]{amsart}

\usepackage{amssymb,amsmath}
\usepackage{amsfonts}

\DeclareMathOperator{\Gr}{Gr}

\DeclareMathOperator{\var}{var}

\newtheorem{theorem}{Theorem}[section]
\newtheorem{proposition}[theorem]{Proposition}
\newtheorem{lemma}[theorem]{Lemma}
\newtheorem{corollary}[theorem]{Corollary}
\theoremstyle{definition}
\newtheorem{example}[theorem]{Example}

\numberwithin{equation}{section}

\begin{document}

\title{Endomorphisms of the lattice of epigroup varieties}
\thanks{Supported by the Ministry of Education and Science of the Russian Federation (project 2248), by a grant of the President of the Russian Federation for supporting of leading scientific schools of the Russian Federation (project 5161.2014.1) and by Russian Foundation for Basic Research (grant 14-01-00524).}

\author{S.\,V.\,Gusev \and B.\,M.\,Vernikov}

\maketitle

\begin{abstract}
We consider epigroups as algebras with two operations (multiplication and pseudoinversion) and construct a countably infinite family of injective endomorphisms of the lattice of all epigroup varieties. An epigroup variety is said to be a \emph{variety of finite degree} if all its nilsemigroups are nilpotent. We characterize epigroup varieties of finite degree in the language of identities and in terms of minimal forbidden subvarieties.
\end{abstract}

\section{Introduction and summary}
\label{intr}

A semigroup $S$ is called an \emph{epigroup} if some power of each element of $S$ lies in a subgroup of $S$. The class of epigroups is quite wide. It includes, in particular, all \emph{completely regular} semigroups (i.e., unions of groups) and all \emph{periodic} semigroups (i.e., semigroups in which every element has an idempotent power). Epigroups have been intensely studied in the literature under different names since the end of 1950's. An overview of results obtained here is given in the fundamental work by L.\,N.\,Shevrin~\cite{Shevrin-94} and his survey~\cite{Shevrin-05}.

It is natural to consider epigroups as \emph{unary semigroups}, i.e., semigroups equipped with an additional unary operation. This operation is defined in the following way. If $S$ is an epigroup and $a\in S$, then some power of $a$ lies in a maximal subgroup of $S$. We denote this subgroup by $G_a$. The unit element of $G_a$ is denoted by $a^\omega$. It is well known (see, e.g.,~\cite{Shevrin-94}) that the element $a^\omega$ is well defined and $aa^\omega=a^\omega a\in G_a$. We denote the inverse of $aa^\omega$ in $G_a$ by $\overline a$. The map $a\mapsto\,\overline a$ is the unary operation on $S$ mentioned above. The element $\overline a$ is called the \emph{pseudoinverse} of $a$. Throughout this article we consider epigroups as algebras with two operations: multiplication and pseudoinversion. In particular, this allows us to speak about varieties of epigroups as algebras with these two operations. An investigation of epigroups in the framework of the theory of varieties was promoted by L.\,N.\,Shevrin in~\cite{Shevrin-94}. An overview of first results obtained here may be found in~\cite[Section~2]{Shevrin-Vernikov-Volkov-09}.{\sloppy

}

It is well known (see, e.g.,~\cite{Shevrin-94,Shevrin-05}) that the class of all epigroups is not a variety. In other words, the variety of unary semigroups generated by this class contains not only epigroups. Denote this variety by $\mathcal{EPI}$. We note that an identity basis of the variety $\mathcal{EPI}$ is known. This result was announced in 2000 by Zhil'tsov~\cite{Zhil'tsov-00}, and its proof was rediscovered recently by Mikhailova~\cite{Mikhailova-13} (some related results can also be found in~\cite{Costa-13}).

If $\mathcal V$ is a semigroup [epigroup] variety then we denote by $\overleftarrow{\mathcal V}$ the variety consisting of all semigroups [epigroups] dual (that is, antiisomorphic) to the semigroups [epigroups] from $\mathcal V$. It is evident that the map $\delta$ of the lattice of all semigroup varieties \textbf{SEM} [the lattice of all epigroup varieties \textbf{EPI}] into itself given by the rule $\delta(\mathcal V)=\overleftarrow{\mathcal V}$ for every variety $\mathcal V$ is an automorphism of this lattice. The question whether there are non-trivial automorphisms of the lattice $\mathbf{SEM}$ [the lattice \textbf{EPI}] different from $\delta$ is still open. We notice that there exist infinitely many non-trivial injective endomorphisms of the lattice \textbf{SEM}. Namely, let $\mathcal V$ be a semigroup variety given by the identities $\{u_i=v_i\mid i\in I\}$, $m$ and $n$ non-negative integers, and $x_1,\dots,x_m,y_1,\dots,y_n$ letters that do not occur in the words $u_i$ and $v_i$ for all $i\in I$. Let $\mathcal V^{m,n}$ be the semigroup variety given by the identities
$$\{x_1\cdots x_mu_iy_1\cdots y_n=x_1\cdots x_mv_iy_1\cdots y_n\mid i\in I\}\ldotp$$
It has been verified in~\cite{Kopamu-03} that $\mathcal V^{m,n}$ does not depend on the choice of an identity basis of the variety $\mathcal V$ and the map $\mathcal{V\mapsto V}^{m,n}$ is an injective endomorphism of the lattice \textbf{SEM}. The first main result of the present paper is an epigroup analogue of this fact.

In order to formulate this result, we need some definitions and notation. We denote by $F$ the free unary semigroup. The unary operation on $F$ will be denoted by $\overline{\phantom x}$. Elements of $F$ are called \emph{unary words} or simply \emph{words}. Let $\Sigma$ be a system of identities written in the language of unary semigroups. Then $K_\Sigma$ stands for the class of all epigroups satisfying $\Sigma$. The class $K_\Sigma$ is not obliged to be a variety because it maybe not closed under taking of infinite Cartesian products (see, e.g.,~\cite[Subsection~2.3]{Shevrin-05} or Example~\ref{not closed} below). A complete classification of identity system $\Sigma$ such that $K_\Sigma$ is a variety is provided by Proposition~\ref{when var} below. If $\Sigma$ has this property, then we will write $\mathbf V[\Sigma]$ along with (and in the same sense as) $K_\Sigma$. We denote the variety of unary semigroups defined by the identity system $\Sigma$ by $\var \Sigma$. Denote the set of all identities that hold in any epigroup by $\Delta$. Thus $\mathcal{EPI}=\var \Delta$. Let $\var_E\Sigma=\mathcal{EPI}\wedge\var \Sigma=\var (\Sigma\cup\Delta)$ (here the symbol $\wedge$ denotes the meet of varieties). Clearly, if the class $K_\Sigma$ is not a variety then $\var_E\Sigma$ contains some unary semigroups that are not epigroups. Moreover, the classes $K_\Sigma$ and $\var_E\Sigma$ may differ even when $K_\Sigma$ is a variety (see Example~\ref{V_Sigma ne var_E Sigma} below). The first main result of the paper is the following

\begin{theorem}
\label{endomorphisms}
Let $\mathcal V=\var_E\{u_i=v_i \mid i\in I\}$ be an epigroup variety, $m$ and $n$ non-negative numbers, and $x_1,\dots,x_m,y_1,\dots,y_n$ letters that do not occur in the words $u_i$ and $v_i$ for all $i\in I$. Put
$$\mathcal V^{m,n}=\var_E\{x_1\cdots x_mu_iy_1\cdots y_n=x_1\cdots x_mv_iy_1\cdots y_n \mid i\in I\}\ldotp$$
Then the variety $\mathcal V^{m,n}$ is an epigroup variety and the map $\mathcal{V\mapsto V}^{m,n}$ is an injective endomorphism of the lattice $\mathbf{EPI}$.
\end{theorem}

An examination of semigroup varieties shows that properties of a variety depend in an essential way on properties of nilsemigroups belonging to the variety. This gives rise to the following definitions. A semigroup variety $\mathcal V$ is called a \emph{variety of finite degree} if all nilsemigroups in $\mathcal V$ are nilpotent. If $\mathcal V$ has a finite degree then it is said to be a \emph{variety of degree} $n$ if the nilpotency degrees of all nilsemigroups in $\mathcal V$ do not exceed $n$ and $n$ is the least number with this property. Semigroup varieties of finite degree and some natural subclasses of this class of varieties were investigated in~\cite{Monzo-Vernikov-11,Sapir-Sukhanov-81,Tishchenko-90,Tishchenko-Volkov-95,Vernikov-08} and other articles, see also~\cite[Section 8]{Shevrin-Sukhanov-89}).

It is well known and may be easily verified that, in a periodic semigroup varieties, pseudoinversion may be expressed via multiplication. Indeed, if an epigroup satisfies the identity
\begin{equation}
\label{x^p=x^{p+q}}
x^p=x^{p+q}
\end{equation}
for some natural numbers $p$ and $q$, then the identity
\begin{equation}
\label{*x=x^{(p+1)q-1}}
\overline x\,=x^{(p+1)q-1}
\end{equation}
holds in this epigroup.
If $p>1$, then the simpler identity
\begin{equation}
\label{*x=x^{pq-1}}
\overline x\,=x^{pq-1}
\end{equation}
is valid. This means that periodic varieties of epigroups may be identified with periodic varieties of semigroups. Semigroup varieties of finite degree are periodic, whence they may be considered as epigroup varieties. It seems to be natural to extend the notions of varieties of finite degree or of degree $n$ to all epigroup varieties. The definitions of epigroup varieties of finite degree or degree $n$ repeat literally the definitions of the same notions for semigroup varieties.

In~\cite[Theorem~2]{Sapir-Sukhanov-81}, semigroup varieties of finite degree were characterized in several ways. In particular, it was proved there that a semigroup variety $\mathcal V$ has a finite degree if and only if it satisfies an identity of the form
\begin{equation}
\label{equ-gen}
x_1\cdots x_n=w
\end{equation}
for some natural $n$ and some word $w$ of length $>n$. Moreover, the proof of this result easily implies that $\mathcal V$ has a degree $\le n$ if and only if it satisfies an identity of the form~\eqref{equ-gen} for some word $w$ of length $>n$. For varieties of degree 2, this equational characterization was made more specific in~\cite[Lemma~3]{Golubov-Sapir-82}. Namely, it was verified there that a semigroup variety has degree $\le2$ if and only if it satisfies an identity of the form $xy=w$ where $w$ is one of the words $x^{m+1}y$, $xy^{m+1}$ or $(xy)^{m+1}$ for some natural $m$. In~\cite[Proposition~2.11]{Vernikov-08}, an analogue of this result of~\cite{Sapir-Sukhanov-81} was obtained for semigroup varieties of degree $\le n$ with arbitrary $n$ (see Proposition~\ref{deg n semi} below). The second objective of this article is to extend the mentioned results of~\cite{Sapir-Sukhanov-81,Vernikov-08} to epigroup varieties. The proof of the corresponding statement makes use of Theorem~\ref{endomorphisms}.

In order to formulate the second main result of this paper, we need some more definitions and notation. A \emph{semigroup word} is a word that does not involve the unary operation.  If $w\in F$, then $\ell(w)$ stands for the length of $w$; we define the length as usual for semigroup words and assume that the length of any non-semigroup word is infinite. If $x$ is a letter and $w$ is a word such that $x$ does not occur in $w$, the pair of identities $wx=xw=w$ is written as the symbolic identity $w=0$. Notice that this notation is justified because a semigroup with such identities has a zero element and all values of the word $w$ in this semigroup are equal to zero. If a system $\Sigma$ of unary identities is such that the class $K_\Sigma$ consists of periodic epigroups (in particular, of nilsemigroups), then $K_\Sigma$ is a periodic semigroup variety, and therefore, is an epigroup variety. Thus, the notation $\mathbf V[\Sigma]$ is well justified in this case. We often use this observation below without any additional reference. Put
\begin{align*}
\mathcal F&=\mathbf V[x^2=0,\,xy=yx],\\
\mathcal F_k&=\mathbf V[x^2=x_1\cdots x_k=0,\,xy=yx]
\end{align*}
where $k$ is an arbitrary natural number. The second main result of the paper is the following

\begin{theorem}
\label{deg fin epi}
For an epigroup variety $\mathcal V$, the following are equivalent:
\begin{itemize}
\item[$1)$]$\mathcal V$ is a variety of finite degree;
\item[$2)$]$\mathcal{V\nsupseteq F}$;
\item[$3)$]$\mathcal V$ satisfies an identity of the form~\eqref{equ-gen} for some natural $n$ and some unary word $w$ with $\ell(w)>n$;
\item[$4)$]$\mathcal V$ satisfies an identity of the form
\begin{equation}
\label{equ-epi}
x_1\cdots x_n=x_1\cdots x_{i-1}\cdot\overline{\overline{x_i\cdots x_j}}\cdot x_{j+1}\cdots x_n
\end{equation}
for some $i$, $j$ and $n$ with $1\le i\le j\le n$.
\end{itemize}
\end{theorem}

As we will see below, the proof of this theorem easily implies the following

\begin{corollary}
\label{deg n epi}
Let $n$ be an arbitrary natural number. For an epigroup variety $\mathcal V$, the following are equivalent:
\begin{itemize}
\item[$1)$]$\mathcal V$ is a variety of degree $\le n$;
\item[$2)$]$\mathcal{V\nsupseteq F}_{n+1}$;
\item[$3)$]$\mathcal V$ satisfies an identity of the form~\eqref{equ-gen} for some unary word $w$ with $\ell(w)>n$;
\item[$4)$]$\mathcal V$ satisfies an identity of the form~\eqref{equ-epi} for some $i$ and $j$ with $1\le i\le j\le n$.
\end{itemize}
\end{corollary}

It is well known that an epigroup variety has degree 1 if and only if it satisfies the identity
\begin{equation}
\label{x=x**}
x=\,\overline{\overline x}
\end{equation}
(see Lemma~\ref{cr} below). Besides that, it is evident that a variety has degree 1 if and only if it does not contain the variety of semigroups with zero multiplication, i.e., the variety $\mathcal F_2$. The equivalence of the claims 1), 2) and 4) of Corollary~\ref{deg n epi} generalizes these known facts.

The paper consists of four sections. Section~\ref{prel} contains definitions, notation and auxiliary results we need. Section~\ref{proof of theorem 1} is devoted to the proof of Theorem~\ref{endomorphisms}, while in Section~\ref{proof of theorem 2} Theorem~\ref{deg fin epi} and Corollary~\ref{deg n epi} are proved.

\section{Preliminaries}
\label{prel}

We denote by $\Gr S$ the set of all group elements of the epigroup $S$. The following well-known fact was verified in~\cite{Munn-61}.

\begin{lemma}
\label{big powers}
If $S$ is an epigroup, $x\in S$ and $x^n\in\Gr S$ for some natural $n$ then $x^m\in\Gr S$ for every $m\ge n$.\qed
\end{lemma}

The next lemma collects several simple and well-known facts (see, e.g.,~\cite{Shevrin-94,Shevrin-05}).

\begin{lemma}
\label{simplest identities}
If $S$ is an epigroup and $x\in S$ then the equalities
\begin{align}
\label{x*x=(x*x)^2=(xx*)}&x\,\overline x\,=(\,x\,\overline x\,)^2=\,\overline{x\,\overline x},\\
\label{xx*=x*x=x^0}&x\,\overline x\,=\,\overline x\,x=x^\omega,\\
\label{x^0x=xx^0=x**}&x^\omega x=xx^\omega=\,\overline{\overline x},\\
\label{x*=(x^2)*x=x(x^2)*}&\overline x\,=\,\overline{x^2}\,x=x\,\overline{x^2},\\
\label{(x^n)*=(x*)^n}&\overline{x^n}\,=\,\overline x\,^n,\\
\label{x***=x*}&\overline{\overline{\overline x}}\,=\,\overline x
\end{align}
hold where $n$ is an arbitrary natural number.\qed
\end{lemma}

The equalities~\eqref{xx*=x*x=x^0} show that the expression $v^\omega$ is well defined in epigroup identities as a short form of the term $v\,\overline v$. So, the equalities~\eqref{x*x=(x*x)^2=(xx*)}--\eqref{x***=x*} are identities valid in arbitrary epigroup. We need the following generalization of~\eqref{xx*=x*x=x^0}.

\begin{corollary}
\label{one more identity}
An arbitrary epigroup satisfies the identities
\begin{equation}
\label{x^n(x*)^n=(x*)^nx^n=x^0}
x^n\,\overline x\,^n=\,\overline x\,^nx^n=x^\omega
\end{equation}
for any natural number $n$.
\end{corollary}

\begin{proof}
Let $S$ be an epigroup and $x\in S$. The identities~\eqref{xx*=x*x=x^0} and the fact that $x^\omega$ is an idempotent in $S$ imply that $x^n\,\overline x\,^n=\,\overline x\,^nx^n=(x\,\overline x\,)^n=(x^\omega)^n=x^\omega$.
\end{proof}

The symbol $\equiv$ denotes the equality relation on $F$. The number of occurrences of multiplication or unary operation in a word $w$ is called the \emph{weight} of $w$.

\begin{lemma}
\label{identity in EPI}
Let $w$ be a non-semigroup word depending on a letter $x$ only. Then the variety $\mathcal{EPI}$ satisfies an identity
\begin{equation}
\label{w=x^p(x^q)*}
w=x^p\,\overline x\,^q
\end{equation}
for some $p\ge 0$ and some positive integer $q$.
\end{lemma}

\begin{proof}
We use induction on the weight of $w$.

\smallskip

\emph{Induction base}. If the weight of $w$ equals 1, then $w\equiv\,\overline x$ and the identity~\eqref{w=x^p(x^q)*} with $p=0$ and $q=1$ holds.

\smallskip

\emph{Induction step}. Suppose that the weight of the word $w$ is $i>1$. Further considerations are divided into two cases.

\smallskip

\emph{Case} 1: $w\equiv w_1w_2$ where $w_1$ and $w_2$ are words with weights less than $i$. Obviously, at least one of the words $w_1$ or $w_2$ involves the unary operation. It suffices to consider the case when $w_1$ is a non-semigroup word. By the induction assumption, the identity $w_1=x^s\,\overline x\,^t$ holds in $\mathcal{EPI}$ for some $s \ge 0$ and some positive integer $t$. If the word $w_2$ involves the unary operation, then, by the induction assumption, the identity $w_2=x^m\,\overline x\,^k$ holds in $\mathcal{EPI}$ for some $m\ge0$ and some $k>0$. If, otherwise, the word $w_2$ is a semigroup one then $w_2\equiv x^r$ for some $r$. In any case, we may apply the identity~\eqref{xx*=x*x=x^0} and conclude that $\mathcal{EPI}$ satisfies an identity of the form~\eqref{w=x^p(x^q)*}.

\smallskip

\emph{Case} 2: $w\equiv\,\overline{w_1}$ where the weight of the word $w_1$ is less than $i$. If $w_1$ is a semigroup word, then $w_1\equiv x^r$ for some $r$. Taking into account the identity~\eqref{(x^n)*=(x*)^n}, we have that the variety $\mathcal{EPI}$ satisfies the identity $w=\,\overline x\,^r$. If, otherwise, the word $w_1$ contains the unary operation then, by the induction assumption, the identity $w_1=x^s\,\overline x\,^t$ holds in $\mathcal{EPI}$ for some $s\ge0$ and some $t>0$. If $s>t$ then
\begin{align*}
w&\equiv\,\overline{w_1}=\,\overline{x^s\,\overline x\,^t}=\,\overline{x^{s-t}x^t\,\overline x\,^t}\,\stackrel{\eqref{xx*=x*x=x^0}}=\,\overline{x^{s-t}}(x\,\overline x\,)^t\,\stackrel{\eqref{x*x=(x*x)^2=(xx*)}}=\,\overline{x^{s-t}(x\,\overline x\,)^{s-t}}\\
&\stackrel{\eqref{xx*=x*x=x^0}}=\,\overline{x^{s-t}(x^\omega)^{s-t}}\,\stackrel{\eqref{x^0x=xx^0=x**}}=\,\overline{(xx^\omega)^{s-t}}\,\stackrel{\eqref{x^0x=xx^0=x**}}=\,\overline{\bigl(\,\overline{\overline x}\,\bigr)^{s-t}}\,\stackrel{\eqref{(x^n)*=(x*)^n}}=\bigl(\,\overline{\overline{\overline x}}\,\bigr)^{s-t}\,\stackrel{\eqref{x***=x*}}=\,\overline x\,^{s-t}\ldotp
\end{align*}
If $s=t$ then
$$w\equiv\,\overline{w_1}=\,\overline{x^s\,\overline x\,^s}\,\stackrel{\eqref{xx*=x*x=x^0}}=\,\overline{(x\,\overline x\,)^s}\,\stackrel{\eqref{x*x=(x*x)^2=(xx*)}}=\,\overline{x\,\overline x}\,\stackrel{\eqref{x*x=(x*x)^2=(xx*)}}=x\,\overline x\,\ldotp$$
Finally, if $s<t$ then
\begin{align*}
w&\equiv\,\overline{w_1}=\,\overline{x^s\,\overline x\,^t}=\,\overline{x^s\,\overline x\,^s\,\overline x\,^{t-s}}\,\stackrel{\eqref{xx*=x*x=x^0}}=\,\overline{(x\,\overline x\,)^s\,\overline x\,^{t-s}}\,\stackrel{\eqref{x*x=(x*x)^2=(xx*)}}=\,\overline{(x\,\overline x\,)^{t-s}\,\overline x\,^{t-s}}\\
&\,\stackrel{\eqref{x*x=(x*x)^2=(xx*)}}=\,\overline{(x\,\overline x\,^2)^{t-s}}\stackrel{\eqref{(x^n)*=(x*)^n}}=\,\overline{\bigl(x\,\overline{x^2}\,\bigr)^{t-s}}\,\stackrel{\eqref{x*=(x^2)*x=x(x^2)*}}=\,\overline{\overline x\,^{t-s}}\stackrel{\eqref{(x^n)*=(x*)^n}}=\bigl(\,\overline{\overline x}\,\bigr)^{t-s}\stackrel{\eqref{x^0x=xx^0=x**}}=(xx^\omega)^{t-s}
\end{align*}
\begin{align*}
&\stackrel{\eqref{xx*=x*x=x^0}}=(x^2\,\overline x\,)^{t-s}\stackrel{\eqref{xx*=x*x=x^0}}=x^{2(t-s)}\,\overline x\,^{t-s}\ldotp
\end{align*}
We have thus proved that the variety $\mathcal{EPI}$ satisfies an identity of the form~\eqref{w=x^p(x^q)*} in any case.
\end{proof}

As usual, we say that an epigroup $S$ has \emph{index} $n$ if $x^n\in\Gr S$ for every $x\in S$ and $n$ is the least number with this property. Following~\cite{Shevrin-94,Shevrin-05}, we denote the class of all epigroups of index $\le n$ by $\mathcal E_n$. It is well known that $\mathcal E_n$ is an epigroup variety (see, e.g.,~\cite[Proposition~6]{Shevrin-94} or~\cite[Proposition~2.10]{Shevrin-05}). An identity $u=v$ is said to be \emph{mixed} if exactly one of $u$ and $v$ is a semigroup word.

\begin{corollary}
\label{finite index}
If a class of unary semigroups $K$ is contained in $\mathcal{EPI}$ and satisfies a mixed identity then $K$ consists of epigroups and $K\subseteq\mathcal E_n$ for some $n$.
\end{corollary}

\begin{proof}
Suppose that $K$ satisfies a mixed identity $u=v$. Substitute some letter $x$ for all letters occurring in this identity. Then we obtain an identity of the form $x^n=w$ for some positive integer $n$ and some non-semigroup word $w$ depending on $x$ only. By Lemma~\ref{identity in EPI}, $\mathcal{EPI}$ satisfies an identity of the form~\eqref{w=x^p(x^q)*}. Therefore, $K$ satisfies
$$x^n=w\,\stackrel{\eqref{w=x^p(x^q)*}}=x^p\,\overline x\,^q\,\stackrel{\eqref{x*=(x^2)*x=x(x^2)*}}=(x^p\,\overline x\,^{q-1})\,\overline x\,^2x\,\stackrel{\eqref{xx*=x*x=x^0}}=(x^p\,\overline x\,^q)x\,\overline x\,\stackrel{\eqref{w=x^p(x^q)*}}=x^nx\,\overline x=x^{n+1}\,\overline x\,\ldotp$$
So, the identity $x^n=x^{n+1}\,\overline x$ holds in $K$. It is well known (see, e.g.,~\cite[p.\,334]{Shevrin-05}) that if a unary semigroup $S\in K$ satisfies this identity then $S$ is an epigroup of index $\le n$.
\end{proof}

If $w\in F$, then $t(w)$ stands for the last letter of $w$.

\begin{lemma}
\label{w=w*z in EPI}
For any word $u$, there is a word $u^\ast$ such that the variety $\mathcal{EPI}$ satisfies the identity $u=u^\ast z$ where $z\equiv t(u)$.
\end{lemma}

\begin{proof}
Let $u$ be an arbitrary word and $z\equiv t(u)$. There are two possible cases: either $u\equiv u^\ast z$ or $u\equiv s_0\,\overline{w_1}$ for some (maybe empty) word $s_0$. In the second case we apply to $u$ the identity~\eqref{x*=(x^2)*x=x(x^2)*} and obtain the word $u_1\equiv s_0\,\overline{w_1^2}\,w_1$ such that the identity $u=u_1$ holds $\mathcal{EPI}$. Here we have two possible cases again: either $w_1 \equiv w^\ast z$ or $w_1\equiv s_1\,\overline{w_2}$. In the second case we again apply the identity~\eqref{x*=(x^2)*x=x(x^2)*} and obtain the word $u_2\equiv s_0\,\overline{w_1^2}\,s_1\,\overline{w_2^2}\,w_2$ such that the identity $u=u_2$ holds in $\mathcal{EPI}$. One can continue this process. It is clear that after finite number of steps we find a word with the required properties.
\end{proof}

In the remaining part of the paper, $u^\ast$ has the same meaning as in Lemma~\ref{w=w*z in EPI}.

As already mentioned, every completely regular semigroup is an epigroup. The operation of pseudoinversion on a completely regular semigroup coincides with the operation of taking the inverse of a given element $x$ in the maximal subgroup that contains $x$. The latter operation is the standard unary operation on the class of completely regular semigroups (see, e.g., the book~\cite{Petrich-Reilly-99} or~\cite[Section~6]{Shevrin-Vernikov-Volkov-09}). Thus, the varieties of completely regular semigroups considered as unary semigroups are epigroup varieties. The following statement is well known.

\begin{lemma}
\label{cr}
For an epigroup variety $\mathcal V$, the following are equivalent:
\begin{itemize}
\item[\textup{a)}]$\mathcal V$ is completely regular;
\item[\textup{b)}]$\mathcal V$ is a variety of degree $1$;
\item[\textup{c)}]$\mathcal V$ satisfies the identity~\eqref{x=x**}.\qed
\end{itemize}
\end{lemma}

The following claim is evident.

\begin{lemma}
\label{nil}
Every nil-epigroup satisfies the identity $\overline x\,=0$.\qed
\end{lemma}

Now we formulate results about semigroup varieties of finite degree obtained in~\cite{Sapir-Sukhanov-81,Vernikov-08}. An identity $u=v$ is called a \emph{semigroup identity} if both $u$ and $v$ are semigroup words. For a semigroup variety $\mathcal V$, the following are equivalent:
\begin{itemize}
\item[a)]$\mathcal V$ is a variety of finite degree;
\item[b)]$\mathcal{V\nsupseteq F}$;
\item[c)]$\mathcal V$ satisfies an identity of the form~\eqref{equ-gen} for some natural $n$ and some semigroup word $w$ with $\ell(w)>n$;
\item[d)]$\mathcal V$ satisfies an identity of the form
\begin{equation}
\label{equ-semi}
x_1\cdots x_n=x_1\cdots x_{i-1}\cdot(x_i\cdots x_j)^{m+1}\cdot x_{j+1}\cdots x_n
\end{equation}
for some $m$, $n$, $i$ and $j$ with $1\le i\le j\le n$.
\end{itemize}
The equivalence of the claims a)--c) was proved in~\cite[Theorem~2]{Sapir-Sukhanov-81}, while the equivalence of a) and d) immediately follows from~\cite[Proposition~2.11]{Vernikov-08}. For varieties of an arbitrary given degree $n$, the following modification of this assertion is valid.

\begin{proposition}
\label{deg n semi}
Let $n$ be a natural number. For a semigroup variety $\mathcal V$, the following are equivalent:
\begin{itemize}
\item[\textup{a)}]$\mathcal V$ is a variety of degree $\le n$;
\item[\textup{b)}]$\mathcal{V\nsupseteq F}_{n+1}$;
\item[\textup{c)}]$\mathcal V$ satisfies an identity of the form~\eqref{equ-gen} for some semigroup word $w$ with $\ell(w)>n$;
\item[\textup{d)}]$\mathcal V$ satisfies an identity of the form~\eqref{equ-semi} for some $m$, $i$ and $j$ with $1\le i\le j\le n$.\qed
\end{itemize}
\end{proposition}

Here the equivalence of the claims a)--c) easily follows from the proof of~\cite[Theorem~2]{Sapir-Sukhanov-81}, while the equivalence of a) and d) is verified in~\cite[Proposition~2.11]{Vernikov-08}.

A semigroup word $w$ is called \emph{linear} if any letter occurs in $w$ at most once. Recall that an identity of the form
$$x_1x_2\cdots x_n=x_{1\pi}x_{2\pi}\cdots x_{n\pi}$$
where $\pi$ is a non-trivial permutation on the set $\{1,2,\dots,n\}$ is called \emph{permutational}. If $w\in F$, then $c(w)$ denotes the set of all letters occurring in the word $w$.

\begin{lemma}
\label{fin deg or permut}
If an epigroup variety $\mathcal V$ satisfies a non-trivial identity of the form~\eqref{equ-gen} then either this identity is permutational or $\mathcal V$ is a variety of degree $\le n$.{\sloppy

}
\end{lemma}

\begin{proof}
If the word $w$ involves the operation of pseudoinversion then every nilsemigroup in $\mathcal V$ satisfies the identity $x_1\cdots x_n=0$ by Lemma~\ref{nil}. Therefore, $\mathcal V$ is a variety of degree $\le n$ in this case. Thus, we may assume that $w$ is a semigroup word. Suppose that $\ell(w)>n$. Then substituting $x$ to all letters occurring in~\eqref{equ-gen}, we obtain an identity of the form~\eqref{x^p=x^{p+q}}. Therefore, $\mathcal V$ is periodic. Then it may be considered as a variety of semigroups. According to Proposition~\ref{deg n semi}, this means that $\mathcal V$ is a variety of degree $\le n$. Suppose now that $\ell(w)\le n$. If $c(w)\ne\{x_1,\dots,x_n\}$ then $x_i\notin c(w)$ for some $1\le i\le n$. One can substitute $x_i^2$ for $x_i$ in~\eqref{equ-gen}. Then we obtain the identity $x_1\cdots x_{i-1}x_i^2x_{i+1}\cdots x_n=w$. Put $w'\equiv x_1\cdots x_{i-1}x_i^2x_{i+1}\cdots x_n$. Then $x_1\cdots x_n=w=w'$ holds in $\mathcal V$. Thus, $\mathcal V$ satisfies the identity $x_1\cdots x_n=w'$ and $\ell(w')>n$. As we have seen above, this implies that $\mathcal V$ has degree $\le n$. Finally, if $c(w)=\{x_1,\dots,x_n\}$ then the fact that $\ell(w)\le n$ implies that $\ell(w)=n$, whence the word $w$ is linear. Therefore, the identity~\eqref{equ-gen} is permutational in this case.
\end{proof}

Put $\mathcal P=\mathbf V[xy=x^2y,\,x^2y^2=y^2x^2]$. We need the following

\begin{lemma}
\label{x_1...x_n=w in P}
If the variety $\mathcal P$ satisfies a non-trivial identity of the form~\eqref{equ-gen} then $n>1$ and $w\equiv w'x_n$ for some word $w'$ with $c(w')=\{x_1,\dots,x_{n-1}\}$.
\end{lemma}

\begin{proof}
It is well known and easy to check that the variety $\mathcal P$ is generated by the semigroup
$$P=\langle a,e\mid e^2=e,\,ea=a,\,ae=0\rangle=\{e,a,0\}\ldotp$$
This semigroup is finite, whence it is an epigroup. Note that $\overline e=e$ and $\overline a=0$. Suppose that $c(w)\ne \{x_1,\dots,x_n\}$. Then there is a letter $x$ that occurs on one side of the identity~\eqref{equ-gen} but does not occur on the other side. Substituting 0 for $x$ and $e$ for all other letters occurring in the identity, we obtain the wrong equality $e=0$. Therefore, $c(w)=\{x_1,\dots,x_n\}$. Substitute now $a$ for $x_n$ and $e$ for all other letters occurring in the identity~\eqref{equ-gen}. The left hand side of the equality we obtain equals $a$. We denote the right hand side of this equality by $b$. Thus, $P$ satisfies the equality $a=b$. If the unary operation applies to the letter $x_n$ in the word $w$ or $t(w)\not\equiv x_n$  then $b=0$. But $a\ne0$ in $P$. Therefore, $w\equiv w'x_n$. If $x_n\in c(w')$ then $b=0$ again, thus $c(w')=\{x_1,\dots,x_{n-1}\}$. Finally, the word $w'$ is non-empty because the identity~\eqref{equ-gen} is non-trivial. Therefore, $n>1$.
\end{proof}

Put $\mathcal C=\mathbf V[x^2=x^3,\,xy=yx]$. The unary semigroup variety generated by an epigroup $S$ is denoted by $\var S$. Clearly, if the semigroup $S$ is finite then $\var S$ is a variety of epigroups. The following statement was formulated without proof in~\cite[Theorem~3.2]{Volkov-00}. We provide the proof here for the sake of completeness.

\begin{proposition}
\label{GrS is right ideal}
Let $\mathcal V$ be an epigroup variety. For an arbitrary epigroup $S\in\mathcal V$, the set $\Gr S$ is a right ideal in $S$ if and only if the variety $\mathcal V$ contains neither $\mathcal C$ nor $\mathcal P$.
\end{proposition}

\begin{proof}
\emph{Necessity}. It is well known and easy to check that the varietiy $\mathcal C$ is generated by the semigroup
$$C=\langle a,e\mid e^2=e,\,ae=ea=a,\,a^2=0\rangle=\{e,a,0\}\ldotp$$
This semigroups is finite, whence it is an epigroup. Let $S$ be one of the epigroups $C$ and $P$. Then $\Gr\,S=\{e,0\}$ and $ea=a\notin\Gr\,S$. We see that $\Gr\,S$ is not a right ideal in $S$, whence $C,P\notin\mathcal V$. Therefore, $\mathcal{C,P\nsubseteq V}$.

\smallskip

\emph{Sufficiency}. Let $S$ be an epigroup in $\mathcal V$ such that $\Gr\,S$ is not a right ideal in $S$. Then there are elements $x\in\Gr\,S$ and $y\in S$ with $xy\notin\Gr\,S$. Put $e=x^\omega$ and $a=xy$. Since $x\in\Gr\,S$, we have $ex=x$, and therefore $ea=exy=xy=a$. Let $A$ be the subepigroup in $S$ generated by the elements $e$ and $a$. Let now $J$ be the ideal in $A$ generated by the element $ae$. If $a\notin J$ then $e\notin J$ because $ea=a$. Hence the Rees quotient epigroup $A/J$ is isomorphic to the epigroup $P$. But this is impossible because $\var\,P=\mathcal{P\nsubseteq V}$. Therefore, $a\in J$, whence $a=baec$ for some $b,c\in A\cup\{1\}$. Suppose that $c\ne e$. In this case $c=ad$ for some $d\in A\cup\{1\}$ because $ea=a$ and the epigroup $A$ satisfies the identities~\eqref{x*=(x^2)*x=x(x^2)*}. Then $a=baec=baead=ba^2d$. According to Lemma~7 of~\cite{Shevrin-94}, $a$ is a group element in $A$, contradicting the choice of the elements $x$ and $y$. It remains to consider the case when $c=e$. Then $ae=(bae)e=bae^2=bae=a$. Let $K$ be the ideal in $A$ generated by the element $a^2$.  Now we apply~\cite[Lemma~7]{Shevrin-94} again and conclude that $a\notin K$. Then the equalities $ea=a$ and $ae=a$ show that the Rees quotient epigroup $A/K$ is isomorphic to the epigroup $C$. But this is not the case because $\var\,C=\mathcal{C\nsubseteq V}$.
\end{proof}

We mention that there is some inaccuracy in the formulation of~\cite[Theorem~3.2]{Volkov-00}. Namely, it contains the words `left ideal' rather than `right ideal'.

Note that semigroup varieties with the property that, for every its member $S$, the set $\Gr S$ is an ideal or a right ideal of $S$ were examined in~\cite{Tishchenko-90}.

An epigroup variety $\mathcal V$ is called a \emph{variety of epigroups with completely regular $n$-th power} if, for every $S\in\mathcal V$, the epigroup $S^n$ is completely regular.

\begin{lemma}
\label{without P and P-dual}
An epigroup variety of degree $\le n$ is a variety of epigroups with completely regular $n$-th power if and only if  it contains neither  $\mathcal P$ nor $\overleftarrow{\mathcal P}$.
\end{lemma}

\begin{proof}
\emph{Necessity}. Let $\mathcal V$ be a variety of epigroups with completely regular $n$-th power. In view of Lemma~\ref{cr} $\mathcal V$ satisfies the identity
\begin{equation}
\label{x_1...x_n=(x_1...x_n)**}
x_1\cdots x_n=\,\overline{\overline{x_1\cdots x_n}}\ldotp
\end{equation}
But Lemma~\ref{x_1...x_n=w in P} and the dual statement imply that this identity is false in the varieties $\mathcal P$ and $\overleftarrow{\mathcal P}$.

\smallskip

\emph{Sufficiency}. Let $\mathcal V$ be a variety of epigroups of degree $\le n$ that contains neither  $\mathcal P$ nor $\overleftarrow{\mathcal P}$. Further, let $S\in\mathcal V$ and $J=\Gr S$. Clearly, the variety $\mathcal C$ is not a variety of finite degree, whence $\mathcal{V\nsupseteq C}$. Thus $\mathcal V$ contains none of the varieties $\mathcal C$, $\mathcal P$ and $\overleftarrow{\mathcal P}$. Now we may apply Proposition~\ref{GrS is right ideal} and the dual statement with the conclusion that $J$ is an ideal in $S$. If $x\in S$ then $x^n\in J$ for some $n$. This means that the Rees quotient $S/J$ is a nilsemigroup. Since $\mathcal V$ is a variety of degree $\le n$, this means that the epigroup $S/J$ satisfies the identity $x_1x_2\cdots x_n=0$. In other words, if $x_1,x_2,\dots,x_n\in S$ then $x_1x_2\cdots x_n\in J$. Therefore, $S^n\subseteq J$, whence the epigroup $S^n$ is completely regular.
\end{proof}

Let $\Sigma$ be a system of identities written in the language of unary semigroups. As we have already noted, the class $K_\Sigma$ is not obliged to be a variety. This claim is confirmed by the following

\begin{example}
\label{not closed}
Put $N_k=\langle a\mid a^{k+1}=0\rangle=\{a,a^2,\dots,a^k,0\}$ for any natural $k$. The semigroup $N_k$ is finite, therefore it is an epigroup. Put
$$N=\prod\limits_{k \in \mathbb N} N_k\ldotp$$
Obviously, the semigroup $N$ is not an epigroup because, for example, no power of the element $(a,\dots,a,\dots)$ belongs to a subgroup. Note that the epigroup $N_k$ is commutative for any $k$. We see that the class $K_\Sigma$ with $\Sigma=\{xy=yx\}$ is not a variety.
\end{example}

If $w$ is a semigroup word, then $\ell_x(w)$ denotes the number of occurrences of the letter $x$ in this word. Recall that a semigroup identity $v=w$ is called \emph{balanced} if $\ell_x(v)=\ell_x(w)$ for any letter $x$. We call an identity $v=w$ \emph{strictly unary} if $v$ and $w$ are non-semigroup words. We say that an identity $v=w$ \emph{follows from an identity system $\Sigma$ in the class of all epigroups} (or \emph{$\Sigma$ implies $v=w$ in the class of all epigroups}) if this identity holds in the class $K_\Sigma$. The following statement gives a complete description of identity systems $\Sigma$ such that $K_\Sigma$ is a variety.

\begin{proposition}
\label{when var}
Let $\Sigma$ be a system of identities written in the language of unary semigroups. The following are equivalent:
\begin{itemize}
\item[$1)$] $K_\Sigma$ is a variety;
\item[$2)$] $\Sigma$ implies in the class of all epigroups some mixed identity;
\item[$3)$] $\Sigma$ contains either a semigroup non-balanced identity or a mixed identity.
\end{itemize}
\end{proposition}

\begin{proof}
1)\,$\longrightarrow$\,3) Suppose that each identity in $\Sigma$ is either balanced or strictly unary. We note that the epigroup $N_k$ from Example~\ref{not closed} satisfies any balanced identity and any strictly unary one. In particular, any identity from $\Sigma$ holds in the epigroup $N_k$. Hence $N_k\in K_\Sigma$ for any $k$. Example~\ref{not closed} shows that the class $K_\Sigma$ is not a variety.

\smallskip

3)\,$\longrightarrow$\,2) The case when $\Sigma$ contains a mixed identity is evident. Suppose now that $\Sigma$ contains a semigroup non-balanced identity $v=w$. Then $\ell_x(u)\ne\ell_x(v)$ for some letter $x$. If $\ell(u)=\ell(v)$ then we substitute $x^2$ to $x$ in $u=v$. As a result, we obtain a semigroup non-balanced identity $u'=v'$ such that $K$ satisfies $u'=v'$, $\ell_x(u')\ne\ell_x(v')$ and $\ell(u')\ne\ell(v')$. This allows us to suppose that $\ell(u)\ne\ell(v)$. Substitute some letter $x$ to all letters occurring in this identity. We obtain an identity of the form~\eqref{x^p=x^{p+q}}. As it was mentioned above, this identity implies in the class of all epigroups the identity~\eqref{*x=x^{(p+1)q-1}}. It remains to note that this identity is mixed.

\smallskip

2)\,$\longrightarrow$\,1) Obviously, the class $K_\Sigma$ is closed under taking of subepigroups and homomorphisms. It remains to prove that it is closed under taking of Cartesian products. Let $\{S_i \mid i \in I\}$ be an arbitrary set of epigroups from $K_\Sigma$. Consider the semigroup
$$S=\prod\limits_{i \in I}S_{i}\ldotp$$
According to Corollary~\ref{finite index}, there exists a number $n$ such that $x^n\in\Gr T$ for any $T\in K_\Sigma$ and any $x\in T$. In particular, the epigroup $S_i$ for any $i\in I$ has this property. But then the semigroup $S$ also satisfies this condition, i.e., $S$ is an epigroup. Obviously, any identity from $\Sigma$ holds in the epigroup $S$. Therefore, $S\in K_\Sigma$ and we are done.
\end{proof}

As it was mentioned above, the classes $K_\Sigma$ and $\var_E\Sigma$ may differ even whenever $K_\Sigma$ is a variety. This claim is confirmed by the following example that is communicated to the authors by V.\,Shaprynski\v{\i}.

\begin{example}
\label{V_Sigma ne var_E Sigma}
Let $\Sigma=\{x=x^2\}$. Consider the two-element semilattice $T=\{e,0\}$. We define on $T$ the unary operation $^\ast$ by the rule $e^\ast=0^\ast=0$. Results of the article~\cite{Mikhailova-13} imply that the variety $\mathcal{EPI}$ has the following identity basis:
\begin{equation}
\label{basis of EPI}
(xy)z=x(yz),\,\overline{xy}\,x=x\,\overline{yx},\,\overline x\,^2x=\,\overline x,\,x^2\,\overline x\,=\,\overline{\overline x},\,\overline{\overline x\,x}\,=\,\overline x\,x,\,\overline{x^p}\,=\,{\overline x}\,^p
\end{equation}
where $p$ runs over the set of all prime numbers. So, any non-trivial identity from $\Delta$ is strictly unary. Therefore, these identities hold in $T$, whence $T\in\mathcal{EPI}\wedge\var \Sigma=\var_E\Sigma$. But $\overline e\,=e$. Therefore, the unary operation $^\ast$ is not the pseudoinversion on $T$, thus $T\notin\mathbf V[\Sigma]$.
\end{example}

Recall that a semigroup identity $u=v$ is called \emph{homotypical} if $c(u)=c(v)$, and \emph{heterotypical} otherwise. The following claim gives a classification of all identity systems $\Sigma$ such that $\mathbf V[\Sigma]=\var_E\Sigma$.

\begin{lemma}
\label{V_Sigma=var_E Sigma}
Let $\Sigma$ be a system of identities written in the language of unary semigroups. The following are equivalent:
\begin{itemize}
\item[\textup{a)}] $\mathbf V[\Sigma]=\var_E\Sigma$;
\item[\textup{b)}] $\var_E\Sigma$ satisfies a mixed identity;
\item[\textup{c)}] $\Sigma$ contains either a semigroup heterotypical identity or a mixed identity.
\end{itemize}
\end{lemma}

\begin{proof}
a)\,$\longrightarrow$\,c) Suppose that each identity in $\Sigma$ is either homotypical or strictly unary. Obviously, the unary semigroup $T$ from Example~\ref{V_Sigma ne var_E Sigma} satisfies all these identities, whence $T\in\var_E\Sigma$. But $T\notin\mathbf V[\Sigma]$, i.e., $\mathbf V[\Sigma]\ne\var_E\Sigma$.

\smallskip

c)\,$\longrightarrow$\,b) If the identity $u=v$ is mixed then the required assertion is obvious. Suppose that the identity $u=v$ is heterotypical. We may assume that there is some letter $x$ that occurs in the word $u$ but does not occur in the word $v$. We substitute $\overline x$ to $x$ in $u=v$. As a result, we obtain a mixed identity.

\smallskip

The implication b)\,$\longrightarrow$\,a) follows from Corollary~\ref{finite index}.
\end{proof}

Suppose that an identity $u=v$ holds in the variety $\var_E\Sigma$. In view of the generally known universal-algebraic considerations, this identity may be obtained from the set of identities $\Sigma\cup\Delta$ by using a finite number of the following operations:

$\bullet$ swap of the left and the right part of the identity,

$\bullet$ equating of two words that are equal to the same word,

$\bullet$ side-by-side multiplication of two identities,

$\bullet$ applying the unary operation to both parts of the identity,

$\bullet$ applying a substitution on $F$ to both parts of the identity.

\noindent For convenience of references, we formulate this fact as a lemma.

\begin{lemma}
\label{deduction in EPI}
Let $\Sigma$ be a system of idenitities written in the language of unary semigroups. If an identity $u=v$ holds in the variety $\var_E \Sigma$ then there exists a sequence of identities
\begin{equation}
\label{sequence}
u_0=v_0,\,u_1=v_1,\,\dots,\,u_m=v_m
\end{equation}
such that the identity $u_0=v_0$ lies in $\Sigma\cup\Delta$, the identity $u_m=v_m$ coincides with the identity $u=v$ and, for every $i=1,\dots,m$, one of the following holds:
\begin{itemize}
\itemsep=0pt
\item[\textup{(i)}] the identity $u_i=v_i$ lies in $\Sigma \cup \Delta$;
\item[\textup{(ii)}] there is $0\le j<i$ such that $u_i\equiv v_j$ and $v_i\equiv u_j$;
\item[\textup{(iii)}] there are $0\le j,k<i$ such that $u_j\equiv u_i$, $v_j\equiv u_k$ and $v_k\equiv v_i$;
\item[\textup{(iv)}] there are $0\le j,k<i$ such that $u_i\equiv u_ju_k$ and $v_i\equiv v_jv_k$;
\item[\textup{(v)}] there is $0\le j<i$ such that $u_i\equiv\,\overline{u_j}$ and $v_i\equiv\,\overline{v_j}$;
\item[\textup{(vi)}] there is $0\le j<i$ such that $u_i\equiv \xi(u_j)$ and $v_i\equiv \xi(v_j)$ for some substitution $\xi$ on $F$.\qed
\end{itemize}
\end{lemma}

Lemma~\ref{deduction in EPI} immediately implies

\begin{corollary}
\label{t(p)=t(q)}
Let a variety $\var_E\Sigma$ satisfies an identity $u=v$. If $t(p)\equiv t(q)$ for any identity $p=q\in\Sigma$ then $t(u)\equiv t(v)$.\qed
\end{corollary}

The sequence of identities~\eqref{sequence} with the properties mentioned in Lemma~\ref{deduction in EPI} is called a \emph{deduction of the identity $v=w$ from the identity system} $\Sigma$.

\begin{lemma}
\label{vz=wz to v=w}
Let $\Theta=\{p_i=q_i\mid i\in I\}$ be an identity system written in the language of unary semigroups and $x$ a letter. Put $\Sigma=\{p_ix=q_ix\mid i\in I\}$. If an identity $ux=vx$ holds in the variety $\var_E\Sigma$, then the identity $u=v$ holds in the variety $\var_E\Theta$.
\end{lemma}

\begin{proof}
Let~\eqref{sequence} be a deduction of the identity $ux=vx$ from the identity system $\Sigma$. Let $1\le i\le m$. Corollary~\ref{t(p)=t(q)} implies that $t(u_i)\equiv t(v_i)$. We are going to verify that the variety $\var_E\Sigma$ satisfies the identity $u_i^\ast=v_i^\ast$. Then, in particular, $\var_E\Sigma$ satisfies $u=v$. We will use induction by $i$. It will be convenient for us to suppose that $\Delta$ is an identity basis of the variety $\mathcal{EPI}$ rather than the set of all identities that hold in this variety. It is clear that, under this assumption, all considerations are valid. Thus, we assume that $\Delta$ coincides with the identity system~\eqref{basis of EPI}.

\smallskip

\emph{Induction base}. If $u_0=v_0\in\Sigma$ then the statement is evident. If $u_0=v_0\in\Delta$, then it may be verified easily that the identity $u_0^\ast=v_0^\ast$ holds in the variety $\mathcal{EPI}$.

\smallskip

\emph{Induction step}. Let now $i>0$. We need consider the cases (i)--(vi) of Lemma~\ref{deduction in EPI}.{\sloppy

}

\smallskip

(i) This case is proved analogously to the induction base.

\smallskip

(ii) The identity $u_i^\ast=v_j^\ast$ holds in the variety $\var_E\Theta$ by the induction assumption. Since the identities $u_i^\ast=v_i^\ast$ and $v_j^\ast=u_j^\ast$ coincide, we are done.

\smallskip

(iii) The identities $u_j^\ast=v_j^\ast$ and $u_k^\ast=v_k^\ast$ (i.e., $u_i^\ast=u_k^\ast$ and $u_k^\ast=v_i^\ast$, respectively) hold in $\var_E\Theta$ by the induction assumption. Therefore, the identity $u_i^\ast=v_i^\ast$ holds in $\var_E\Theta$ as well.

\smallskip

(iv) Note that $u_i^\ast\equiv u_ju_k^\ast$ and $v_i^\ast\equiv v_j v_k^\ast$. Let $y\equiv t(u_j)\equiv t(v_j)$. Using the induction assumption, we conclude that the identities $u_j^\ast=v_j^\ast$ and $u_k^\ast=v_k^\ast$ hold in $\var_E\Theta$. Multiplying the former identity on a letter $y$ on the right, we see that $\var_E\Theta$ satisfies the identity $u_j^\ast y=v_j^\ast y$. Since the variety $\mathcal{EPI}$ satisfies the identities $u_j=u_j^\ast y$, $v_j=v_j^\ast y$ and $\var_E\Theta\subseteq\mathcal{EPI}$, we have that $\var_E\Theta$ satisfies the identity $u_j=v_j$. Multiplying this identity and the identity $u_k^\ast=v_k^\ast$, we conclude that $\var_E\Theta$ satisfies the identity $u_i^\ast=v_i^\ast$.

\smallskip

(v) The identity $u_j^\ast=v_j^\ast$ holds in the variety $\var_E\Theta$ by the induction assumption. Note that $u_i^\ast\equiv\,\overline{u_j}\,^2u_j^\ast$ and $v_i^\ast\equiv\,\overline{v_j}\,^2v_j^\ast$. As we have seen in the case~(iv), $\var_E\Theta$ satisfies the identity $u_j=v_j$. It is evident that the identity $\overline{u_j}\,^2u_j^\ast=\,\overline{v_j}\,^2v_j^\ast$ may be deduced from the identities $u_j=v_j$ and $u_j^\ast=v_j^\ast$. Hence $\var_E\Theta$ satisfies $u_i^\ast=v_i^\ast$.

\smallskip

(vi) As usual, $u_j^\ast=v_j^\ast$ holds in $\var_E\Theta$ by the induction assumption. Let $t(u_j)\equiv t(v_j)\equiv x$. Then $\bigl(\xi(u_j)\bigr)^\ast\equiv\xi(u^\ast_j)\bigl(\xi(x)\bigr)^\ast$ and $\bigl(\xi(v_j)\bigr)^\ast\equiv\xi(v^\ast_j)\bigl(\xi(x)\bigr)^\ast$. This implies that $\var_E\Theta$ satisfies the identity $u_i^\ast=v_i^\ast$.
\end{proof}

\section{Endomorphisms of the lattice EPI}
\label{proof of theorem 1}

To verify Theorem~\ref{endomorphisms}, we need some auxiliary facts.

\begin{lemma}
\label{letter on right}
Let $\Sigma=\{p_\alpha=q_\alpha\mid\alpha\in\Lambda\}$. If the variety $\var_E\Sigma$ satisfies an identity $u=v$ and $x$ is a letter that does not occur in the words $u,v,p_\alpha$ and $q_\alpha$ \textup(for all $\alpha\in\Lambda$\textup), then the identity $ux=vx$ follows from the identity system $\Sigma'=\{p_\alpha x=q_\alpha x\mid\alpha\in\Lambda\}$ in the class of all epigroups.
\end{lemma}

\begin{proof}
Let~\eqref{sequence} be a deduction of the identity $u=v$ from the identity system $\Sigma$. Let $1\le i\le m$. Corollary~\ref{t(p)=t(q)} implies that $t(u_i)\equiv t(v_i)$. Let $y$ be a letter with $y\not\equiv x$. If the letter $x$ occurs in some identities of the sequence~\eqref{sequence} then we substitute $y$ to $x$ in all such identities. The identities from $\Sigma\cup\{u=v\}$ will not change because these identities do not contain the letter $x$, and the identities from $\Delta$ will still remain in $\Delta$. The sequence we obtain is a deduction of the identity $u=v$ from the identity system $\Sigma\cup\Delta$ again, and all the identities of this deduction do not contain the letter $x$. We may assume without any loss that already the deduction~\eqref {sequence} possesses the last property.

For each $i=0,1,\dots,m$, the identity $u_i=v_i$ holds in the variety $\var_E\Sigma$. Since the identity $u_m=v_m$ coincides with the identity $u=v$, it suffices to verify that, for each $i=0,1,\dots,m$, the identity $u_ix=v_ix$ follows from the identity system $\Sigma'$ in the class of all epigroups. The proof of this claim is given by induction on $i$.

\smallskip

\emph{Induction base} is evident because the identity $u_0=v_0$ lies in $\Sigma\cup\Delta$.

\smallskip

\emph{Induction step}. Let now $i>0$. One can consider the cases (i)--(vi).

\smallskip

(i) This case is obvious.

\smallskip

(ii) By the induction assumption, the identity $u_jx=v_jx$ follows from the identity system $\Sigma'$ in the class of all epigroups. Since the identity $u_ix=v_ix$ coincides with the identity $v_jx=u_jx$, we are done.

\smallskip

(iii) By the induction assumption, the identities $u_jx=v_jx$ (i.e., $u_ix=u_kx$) and $u_kx=v_kx$ (i.e., $u_kx=v_ix$) follow from the identity system $\Sigma'$ in the class of all epigroups. Therefore, the identity $u_ix=v_ix$ follows from the identity system $\Sigma'$ in the class of all epigroups too.

\smallskip

(iv) By the induction assumption, the identities $u_jx=v_jx$ and $u_kx=v_kx$ follow from the identity system $\Sigma'$ in the class of all epigroups. We substitute $u_kx$ to $x$ in the identity $u_jx=v_jx$. Since the letter $x$ does not occur in the words $u_j$ and $v_j$, we obtain the identity $u_ju_kx=v_ju_kx$, i.e., $u_ix=v_ju_kx$. Further, we multiply the identity $u_kx=v_kx$ on $v_j$ from the left. Here we obtain the identity $v_ju_kx=v_jv_kx$, i.e., $v_ju_kx=v_ix$. We see that the identity system $\Sigma'$ implies the identities $u_ix=v_ju_kx$ and $v_ju_kx=v_ix$ in the class of all epigroups, whence the identity $u_ix=v_ix$ also follows from $\Sigma'$ in the class of all epigroups.

\smallskip

(v) By the induction assumption, the identity $u_jx=v_jx$ follows from the identity system $\Sigma'$ in the class of all epigroups. Since $u_i\equiv\,\overline{u_j}$ and $v_i\equiv\,\overline{v_j}$, it remains to verify that the identity $\overline{u_j}\,x=\,\overline{v_j}\,x$ follows from the identity system $\Sigma'$ in the class of all epigroups. Suppose that an epigroup $S$ satisfies the identity $u_j x=v_j x$ and $\bigl|c(u_j)\cup c(v_j)\bigr|=k$. We fix arbitrary elements $a_1,\dots,a_k$ and $b$ in $S$. Put $U_j=u_j(a_1,\dots,a_k)$ and $V_j=v_j(a_1,\dots,a_k)$. Then
\begin{align}
\label{U_jb=V_jb}
U_jb=V_jb\ldotp
\end{align}
We need to verify that $\overline{U_j}\,b=\,\overline{V_j}\,b$. First of all, we verify that
\begin{equation}
\label{V_j^{s+1}=U_j^sV_j}
V_j^{s+1}=U_j^sV_j
\end{equation}
for any natural $s$. We use induction by $s$. If $s=1$ then the equality~\eqref{V_j^{s+1}=U_j^sV_j} coincides with~\eqref{U_jb=V_jb} where $b=V_j$. If $s>1$ then
\begin{align*}
V_j^{s+1}&=V_jV_j^s\\
&=U_jV_j^s&&\text{by~\eqref{U_jb=V_jb} with }b=V_j^s\\
&=U_jU_j^{s-1}V_j&&\text{by the induction assumption}\\
&=U_j^sV_j,&&
\end{align*}
and the equality~\eqref{V_j^{s+1}=U_j^sV_j} is proved. The equality~\eqref{U_jb=V_jb} with $b=V_j$ and~\eqref{xx*=x*x=x^0} imply that $U_j^\omega V_j=\,\overline{U_j}\,U_jV_j=\,\overline{U_j}\,V_j^2$. Thus,
\begin{equation}
\label{U_j*V_j^2=U_j^0V_j}
\overline{U_j}\,V_j^2=U_j^\omega V_j\ldotp
\end{equation}
Let now $s$ be a natural number with $s\ge2$. Using~\eqref{U_j*V_j^2=U_j^0V_j}, we have
$$\overline{U_j}\,\rule{0pt}{10pt}^sV_j\rule{0pt}{10pt}^s=\,\overline{U_j}\,\rule{0pt}{10pt}^{s-1}\bigl(\,\overline{U_j}\,V_j\rule{0pt}{10pt}^2\bigr)V_j\rule{0pt}{10pt}^{s-2}=\,\overline{U_j}\,\rule{0pt}{10pt}^{s-1}U_j\rule{0pt}{10pt}^\omega V_jV_j\rule{0pt}{10pt}^{s-2}=\,\overline{U_j}\,\rule{0pt}{10pt}^{s-1}V_j\rule{0pt}{10pt}^{s-1}\ldotp$$
Therefore, $\overline{U_j}\,\rule{0pt}{10pt}^sV_j\rule{0pt}{10pt}^s=\,\overline{U_j}\,\rule{0pt}{10pt}^{s-1}V_j\rule{0pt}{10pt}^{s-1}=\cdots=\,\overline{U_j}\,V_j$. Thus,
\begin{equation}
\label{(U_j*)^sV_j^s=U_j*V_j}
\overline{U_j}\,\rule{0pt}{10pt}^sV_j\rule{0pt}{10pt}^s=\,\overline{U_j}\,V_j
\end{equation}
for any natural $s$. Since $S$ is an epigroup, there are numbers $g$ and $h$ such that $U_j^g,V_j^h\in\Gr S$. Put $m=\max\{g,h\}$. For any $s\ge m$ we have
\begin{align*}
U_j^\omega V_j^s&=U_j^\omega(V_j^sV_j^\omega)&&\text{because }V_j^s\in G_{V_j}\text{ by Lemma~\ref{big powers}}\\
&=U_j^\omega(V_j^{s+1}\,\overline{V_j}\,)&&\text{by~\eqref{xx*=x*x=x^0}}\\
&=(U_j^\omega U_j^s)V_j\,\overline{V_j}&&\text{by~\eqref{V_j^{s+1}=U_j^sV_j}}\\
&=(U_j^sV_j)\,\overline{V_j}&&\text{because }U_j^s\in G_{U_j}\text{ by Lemma~\ref{big powers}}\\
&=V_j^{s+1}\,\overline{V_j}&&\text{by~\eqref{V_j^{s+1}=U_j^sV_j}}\\
&=V_j^sV_j^\omega&&\text{by~\eqref{xx*=x*x=x^0}}\\
&=V_j^s&&\text{because }V_j^s\in G_{V_j}\text{ by Lemma~\ref{big powers}}\ldotp
\end{align*}
Thus,
\begin{equation}
\label{U_j^0V_j^s=V_j^s}
U_j^\omega V_j^s=V_j^s
\end{equation}
for any $s\ge m$. Note also that
\begin{align*}
U_j^\omega V_j&=\,\overline{U_j}\,\rule{0pt}{10pt}^mU_j\rule{0pt}{10pt}^mV_j&&\text{by~\eqref{x^n(x*)^n=(x*)^nx^n=x^0}}\\
&=\,\overline{U_j^m}\,(U_j^mV_j)&&\text{by~\eqref{(x^n)*=(x*)^n}}\\
&=\,\overline{U_j^m}\,V_j^{m+1}&&\text{by~\eqref{V_j^{s+1}=U_j^sV_j}}\\
&=\,\overline{U_j^m}\,V_j^{m+1}V_j^\omega&&\text{because }V_j^{m+1}\in G_{V_j}\text{ by Lemma~\ref{big powers}}\\
&=\,\overline{U_j^m}\,V_j^{m+1}(V_j^m)^\omega&&\text{because }G_{V_j}=G_{V_j^m}\\
&=\,\overline{U_j^m}\,V_j^{m+1}V_j^m\,\overline{V_j^m}&&\text{by~\eqref{xx*=x*x=x^0}}\\
&=\bigl(\,\overline{U_j}\,\rule{0pt}{10pt}^mV_j\rule{0pt}{10pt}^m\bigr)V_j\rule{0pt}{10pt}^{m+1}\,\overline{V_j^m}&&\text{by~\eqref{(x^n)*=(x*)^n}}\\
&=\,\overline{U_j}\,V_jV_j^{m+1}\,\overline{V_j^m}&&\text{by~\eqref{(U_j*)^sV_j^s=U_j*V_j}}\\
&=(\,\overline{U_j}\,V_j^2)(V_j^m\,\overline{V_j^m}\,)&&\\
&=(U_j^\omega V_j)(V_j^m\,\overline{V_j^m}\,)&&\text{by~\eqref{U_j*V_j^2=U_j^0V_j}}\\
&=V_j^{m+1}\,\overline{V_j^m}&&\text{by~\eqref{U_j^0V_j^s=V_j^s}}\\
&=V_jV_j\rule{0pt}{10pt}^m\,\overline{V_j}\,\rule{0pt}{10pt}^m&&\text{by~\eqref{(x^n)*=(x*)^n}}\\
&=V_jV_j^\omega&&\text{by~\eqref{x^n(x*)^n=(x*)^nx^n=x^0}}\\
&=\,\overline{\overline{V_j}}&&\text{by~\eqref{x^0x=xx^0=x**}}\ldotp
\end{align*}
Thus,
\begin{equation}
\label{U_j^0V_j=V_j**}
U_j^\omega V_j=\,\overline{\overline{V_j}}\ldotp
\end{equation}
Besides that,
\begin{align*}
\overline{U_j}\,V_j^\omega&=\bigl(\,\overline{U_j}\,V_j\rule{0pt}{10pt}^2\bigr)\,\overline{V_j}\,\rule{0pt}{10pt}^2&&\text{by~\eqref{x^n(x*)^n=(x*)^nx^n=x^0}}\\
&=(U_j\rule{0pt}{10pt}^\omega V_j)\,\overline{V_j}\,\rule{0pt}{10pt}^2&&\text{by~\eqref{U_j*V_j^2=U_j^0V_j}}\\
&=\,\overline{\overline{V_j}}\;\overline{V_j}\,\rule{0pt}{10pt}^2&&\text{by~\eqref{U_j^0V_j=V_j**}}\\
&=V_j^\omega\,\overline{V_j}&&\text{because }\overline{\overline{V_j}}\text{ and }\overline{V_j}\text{ are mutually inverse in }G_{V_j}\\
&=\,\overline{V_j}&&\text{because }\overline{V_j}\,\in G_{V_j}\ldotp
\end{align*}
Thus,
\begin{equation}
\label{U_j*V_j^0=V_j*}
\overline{U_j}\,V_j^\omega=\,\overline{V_j}\ldotp
\end{equation}
Finally, we have
\begin{align*}
\overline{U_j}\,b&=\,\overline{U_j^2}\,(U_jb)&&\text{by~\eqref{x*=(x^2)*x=x(x^2)*}}\\
&=\,\overline{U_j^2}\,(V_jb)&&\text{by~\eqref{U_jb=V_jb}}\\
&=\,\overline{U_j^2}\,(U_j^\omega V_j)b&&\text{because }\overline{U_j^2}\,\in G_{U_j}\\
&=\,\overline{U_j^2}\,\bigl(\,\overline{\overline{V_j}}\,b\bigr)&&\text{by~\eqref{U_j^0V_j=V_j**}}\\
&=(\,\overline{U_j^2}\,V_j^\omega)V_jb&&\text{by~\eqref{x^0x=xx^0=x**}}\\
&=\,\overline{U_j}\,\rule{0pt}{10pt}^2V_j\rule{0pt}{10pt}^\omega V_jb&&\text{by~\eqref{(x^n)*=(x*)^n}}\\
&=\,\overline{U_j}\;\overline{V_j}\,V_jb&&\text{by~\eqref{U_j*V_j^0=V_j*}}\\
&=\,\overline{U_j}\,V_j^\omega\,\overline{V_j}\,V_jb&&\text{because }\overline{V_j}\,\in G_{V_j}\\
&=\,\overline{U_j}\,V_j^\omega\,\overline{V_j}\,V_jb&&\text{because }\overline{V_j}\,\in G_{V_j}\\
&=\,\overline{V_j}\,(\,\overline{V_j}\,V_j)b&&\text{by~\eqref{U_j*V_j^0=V_j*}}\\
&=\,\overline{V_j}\,V_j^\omega b&&\text{by~\eqref{xx*=x*x=x^0}}\\
&=\,\overline{V_j}\,b&&\text{because }\overline{V_j}\,\in G_{V_j}\ldotp
\end{align*}
We prove that $\overline{U_j}\,b=\,\overline{V_j}\,b$. This completes a consideration of the case (v).

\smallskip

(vi) By the induction assumption, the identity $u_jx=v_jx$ follows from the identity system $\Sigma'$ in the class of all epigroups. We may assume without any loss that $c(u_j)\cup c(v_j)=\{x_1,\dots,x_k\}$ and the identity $u_i=v_i$ is obtained from the identity $u_j=v_j$ by a substitution of some word $w$ for some letter that occurs in the identity $u_j=v_j$. Since $x\notin c(u_i)\cup c(v_i)$, the letter $x$ does not occur in the word $w$. We substitute $w$ to $x$ in the identity $u_jx=v_jx$. Then we obtain the identity $u_ix=v_ix$. Therefore, this identity follows from the identity system $\Sigma'$ in the class of all epigroups.
\end{proof}

\begin{lemma}
\label{v=w to vz=wz}
Let $\Theta=\{p_i=q_i\mid i\in I\}$ be an identity system such that the variety $\var_E\Theta$ is a variety of epigroups, and let $x$ be a letter with $x\notin c(p_i)\cup c(q_i)$ for all $i\in I$. Put $\Sigma=\{p_ix=q_ix\mid i\in I\}$. Then the variety $\var_E\Sigma$ is a variety of epigroups and the identity $ux=vx$ holds in $\var_E\Sigma$ whenever the identity $u=v$ holds in $\var_E\Theta$.
\end{lemma}

\begin{proof}
Let $y$ be a letter with $y\notin c(u)\cup c(v)$ and $y\notin c(p_i)\cup c(q_i)$ for all $i\in I$. By Lemma~\ref{letter on right}, the identity $uy=vy$ holds in the class of epigroups $K_\Sigma$. Substituting $x$ for $y$ in this identity, we obtain that the identity $ux=vx$ is valid in $K_\Sigma$. Now Lemma~\ref{V_Sigma=var_E Sigma} applies with the conclusion that the system $\Theta$ contains an identity $p_i=q_i$ that is either mixed or semigroup heterotypical. Then the identity $p_ix=q_ix \in \Sigma$ also has one of these two properties. In view of Proposition~\ref{when var}, the class $K_\Sigma$ is an epigroup variety. Now Lemma~\ref{V_Sigma=var_E Sigma} applies again, and we conclude that $K_\Sigma=V[\Sigma]=\var_E\Sigma$.
\end{proof}

\begin{corollary}
\label{var{u=v}=var{p=q}}
Let $\Theta_1=\{v_i=w_i\mid i\in I\}$, $\Theta_2=\{p_\alpha=q_\alpha\mid\alpha\in\Lambda\}$ and $x$ a letter such that $x\notin c(v_i)\cup c(w_i)$ for all $i\in I$ and $x\notin c(p_\alpha)\cup c(q_\alpha)$ for all $\alpha\in\Lambda$. Put $\Sigma_1=\{v_ix=w_ix\mid i\in I\}$ and $\Sigma_2=\{p_\alpha x=q_\alpha x\mid\alpha\in\Lambda\}$. If $\var_E\Theta_1$ is an epigroup variety then $\var_E\Theta_1=\var_E\Theta_2$ if and only if $\var_E\Sigma_1=\var_E\Sigma_2$.
\end{corollary}

\begin{proof}
\emph{Necessity}. Let $\var_E\Theta_1=\var_E\Theta_2$. This means that the variety $\var_E \Theta_1$ satisfies the identity $p_\alpha=q_\alpha$ for each $\alpha \in \Lambda$. Lemma~\ref{v=w to vz=wz} implies that the identity $p_\alpha x=q_\alpha x$ holds in the variety $\var_E\Sigma_1$. Therefore, $\var_E\Sigma_1\subseteq\var_E\Sigma_2$. The opposite inclusion holds as well by symmetry. Thus $\var_E\Sigma_1=\var_E\Sigma_2$.

\smallskip

\emph{Sufficiency.} Let now $\var_E\Sigma_1=\var_E\Sigma_2$. Then $p_\alpha x=q_\alpha x$ holds in $\var_E\Sigma_1$ for each $\alpha\in\Lambda$. Lemma~\ref{vz=wz to v=w} then implies that the variety $\var_E \Theta_1$ satisfies all identities $p_\alpha=q_\alpha$, whence $\var_E\Theta_1\subseteq\var_E\Theta_2$. By symmetry, $\var_E\Theta_2\subseteq\var_E\Theta_1$. Therefore, $\var_E\Theta_1=\var_E\Theta_2$.
\end{proof}

Now we are ready to complete the proof of Theorem~\ref{endomorphisms}. By symmetry and evident induction, it suffices to verify that if $x$ is a letter with $x\notin c(u_i)\cup c(v_i)$ for all $i\in I$ then $\mathcal V^{0,1}=\var_E\{u_ix=v_ix \mid i\in I\}$ is an epigroup variety and the map $f$ from $\mathbf{EPI}$ into itself given by the rule $f(\mathcal{V)=V}^{0,1}$ is an injective endomorphism of the lattice $\mathbf{EPI}$.

The claim that $\mathcal V^{0,1}$ is an epigroup variety follows from Lemma~\ref{v=w to vz=wz}. The map $f$ is well defined and injective by Corollary~\ref{var{u=v}=var{p=q}}. It remains to verify that the map $f$ is a homomorphism. Let $\mathcal V_1$ and $\mathcal V_2$ be epigroup varieties. Further, let $\mathcal V_1=\var_E\Theta_1$ and $\mathcal V_2=\var_E\Theta_2$ for identity systems $\Theta_1=\{u_i=v_i\mid i\in I\}$ and $\Theta_2=\{p_\alpha=q_\alpha\mid \alpha\in\Lambda\}$. Suppose that $x$ is a letter such that $x\notin c(u_i)\cup c(v_i)$ for all $i\in I$ and $x\notin c(p_\alpha)\cup c(q_\alpha)$ for all $\alpha\in\Lambda$. Put $\Sigma_1=\{u_ix=v_ix\mid i\in I\}$ and $\Sigma_2=\{p_\alpha x=q_\alpha x\mid\alpha\in\Lambda\}$. Then we have
\begin{align*}
f(\mathcal V_1\wedge\mathcal V_2)&=f(\var_E\Theta_1\wedge\var_E\Theta_2)=f(\var_E(\Theta_1\cup\Theta_2))\\
&{}=\var_E(\Sigma_1\cup\Sigma_2)=\var_E\Sigma_1\wedge\var_E\Sigma_2=f(\mathcal V_1)\wedge f(\mathcal V_2)\ldotp
\end{align*}
Thus, $f(\mathcal V_1\wedge\mathcal V_2)=f(\mathcal V_1)\wedge f(\mathcal V_2)$. It remains to verify that $f(\mathcal V_1\vee\mathcal V_2)=f(\mathcal V_1)\vee f(\mathcal V_2)$.

Let $\mathcal V_1\vee\mathcal V_2=\var_E\Theta$ where $\Theta=\{s_j=t_j\mid j\in J\}$. We may assume without loss of generality that $x\notin c(s_j)\cup c(y_j)$ for all $j\in J$ because we may rename letters otherwise. Further, $\var_E\Theta_1=\var_E(\Theta_1\cup \Theta)$ because $\mathcal V_1\subseteq\mathcal V_1\vee\mathcal V_2$. Now Corollary~\ref{var{u=v}=var{p=q}} applies with the conclusion that
\begin{align*}
f(\mathcal V_1)=f(\var_E(\Theta_1\cup\Theta))&=\var_E\{u_ix=v_ix,\,s_jx=t_jx\mid i\in I,j\in J\}\\
&\subseteq\var_E\{s_jx=t_jx\mid j\in J\}=f(\mathcal V_1\vee\mathcal V_2)\ldotp
\end{align*}
Analogously, $f(\mathcal V_2)\subseteq f(\mathcal V_1\vee\mathcal V_2)$. Therefore, $f(\mathcal V_1)\vee f(\mathcal V_2)\subseteq f(\mathcal V_1\vee\mathcal V_2)$.

It remains to verify the opposite inclusion. Let the identity $u=v$ hold in the variety $f(\mathcal V_1)\vee f(\mathcal V_2)$. Then it holds in $f(\mathcal V_i)$ with $i=1,2$. We may assume without any loss that $t(u)\equiv x$. Applying Corollary~\ref{t(p)=t(q)}, we conclude that $t(v)\equiv t(u)\equiv x$. Lemma~\ref{w=w*z in EPI} implies that the identities $u=u^\ast x$ and $v=v^\ast x$ hold in the variety $\mathcal{EPI}$. Hence, the variety $f(\mathcal V_1)$ satisfies the identity $u^\ast x =v^\ast x$. Now Lemma~\ref{vz=wz to v=w} applies and we conclude that the variety $\mathcal V_1$ satisfies the identity $u^\ast=v^\ast$. Analogous considerations show that this identity is true in the variety $\mathcal V_2$ as well. Thus, $u^\ast=v^\ast$ holds in the variety $\mathcal V_1\vee\mathcal V_2$. The letter $x$ does not occur in any identity from $\Theta$. Now Lemma~\ref{v=w to vz=wz} applies with the conclusion that the variety $f(\mathcal V_1\vee\mathcal V_2)$ satisfies the identity $u^\ast x=v^\ast x$. Then this variety satisfies the identity $u=v$ too. Therefore, $f(\mathcal V_1\vee\mathcal V_2)\subseteq f(\mathcal V_1)\vee f(\mathcal V_2)$.\qed

\section{Varieties of finite degree}
\label{proof of theorem 2}

The aim of this section is to prove Theorem~\ref{deg fin epi} and Corollary~\ref{deg n epi}. The implication 4)\,$\longrightarrow$\,3) of Theorem~\ref{deg fin epi} is obvious, while the implication 3)\,$\longrightarrow$\,2) follows from the evident fact that the variety $\mathcal F$ does not satisfy any identity of the form~\eqref{equ-gen} with $\ell(w)>n$. It remains to verify the implications 1)\,$\longrightarrow$\,4) and 2)\,$\longrightarrow$\,1).

1)\,$\longrightarrow$\,4) We are going to verify that if an epigroup variety $\mathcal V$ is a variety of degree $\le n$ then it satisfies an identity of the form~\eqref{equ-epi} for some $i$ and $j$ with $1\le i\le j\le n$. Clearly, this implies the implication. We use induction on $n$.

\smallskip

\emph{Induction base}. If $\mathcal V$ is a variety of degree 1, then it satisfies the identity of the form~\eqref{equ-epi} with $i=j=n=1$ by Lemma~\ref{cr}.

\smallskip

\emph{Induction step}. Let $n>1$ and $\mathcal V$ is a variety of degree $\le n$. If $\mathcal P,\overleftarrow{\mathcal P}\nsubseteq\mathcal V$ then $\mathcal V$ is a variety of epigroups with completely regular $n$th power by Lemma~\ref{without P and P-dual}. By Lemma~\ref{cr} $\mathcal V$ then satisfies the identity~\eqref{x_1...x_n=(x_1...x_n)**}, i.e., the identity of the form~\eqref{equ-epi} with $i=1$ and $j=n$. Suppose now that $\mathcal V$ contains one of the varieties $\mathcal P$ or $\overleftarrow{\mathcal P}$. We will assume without loss of generality that $\mathcal{P\subseteq V}$.

The variety $\mathcal F_{n+1}$ has degree $n+1$, whence $\mathcal{V\nsupseteq F}_{n+1}$. Therefore, there is an identity $u=v$ that holds in $\mathcal V$ but is false in $\mathcal F_{n+1}$. In view of Lemma~\ref{nil}, every non-semigroup word equals to 0 in $\mathcal F_{n+1}$. It is evident that every non-linear semigroup word and every semigroup word of length $>n$ equal to 0 in $\mathcal F_{n+1}$ as well. Therefore, we may assume without any loss that $u$ is a linear semigroup word of length $\le n$, i.e., $u\equiv x_1\cdots x_m$ for some $m\le n$. Since $\mathcal{P\subseteq V}$, the identity $x_1\cdots x_m=v$ holds in $\mathcal P$. Now Lemma~\ref{x_1...x_n=w in P} applies with the conclusion that $m>1$ and $v\equiv v'x_m$ for some word $v'$ with $c(v')=\{x_1,\dots,x_{m-1}\}$. Suppose that $\ell(v')\le m-1$. In particular, this means that $v'$ is a semigroup word. Since $c(v')=\{x_1,\dots,x_{m-1}\}$, we have that $\ell(v')=m-1$. Therefore, the word $v'$ is linear, whence $v$ is linear too. This means that $u=v$ is a permutational identity. But every permutational identity holds in the variety $\mathcal F_{n+1}$, while the identity $u=v$ is false in $\mathcal F_{n+1}$. Hence $\ell(v')>m-1$. So, taking into account Proposition~\ref{when var}, we have that $\mathcal V \subseteq \mathbf V[x_1\cdots x_{m}=v'x_m]$.

Also, Proposition~\ref{when var} implies that the class of epigroups satisfying the identity
\begin{equation}
\label{x_1...x_{m-1}=v'}
x_1\cdots x_{m-1}=v'
\end{equation}
is a variety. We denote this variety by $\mathcal V'$. According to Lemma~\ref{fin deg or permut}, $\mathcal V'$ is a variety of degree $\le m-1$. Since $m\le n$, we use the induction assumption and conclude that
$$\mathcal V' \subseteq \var_E\{x_1\cdots x_{m-1}=x_1\cdots x_{i-1}\cdot\overline{\overline{x_i\cdots x_j}}\cdot x_{j+1}\cdots x_{m-1}\}$$
for some $1\le i\le j\le m-1$. Further considerations are divided into two cases.

\smallskip

\emph{Case} 1: the word $v'$ involves the unary operation. According to Lemma~\ref{V_Sigma=var_E Sigma}, $\mathcal V'= \mathbf V[x_1\cdots x_{m-1}=v']=\var_E\{x_1\cdots x_{m-1}=v'\}$.
The letter $x_m$ does not occur in any of the words $x_1\cdots x_{m-1}$, $v'$ and $x_1\cdots x_{i-1}\cdot\overline{\overline{x_i\cdots x_j}}\cdot x_{j+1}\cdots x_{m-1}$. Now Theorem~\ref{endomorphisms} applies with the conclusion that
$$\var_E\{x_1\cdots x_{m}=v'x_m\}\subseteq \var_E\{x_1\cdots x_m=x_1\cdots x_{i-1}\cdot\overline{\overline{x_i\cdots x_j}}\cdot x_{j+1}\cdots x_m\}\ldotp$$
Therefore, $\mathcal V$ satisfies the identity
\begin{equation}
\label{short equ-epi}
x_1\cdots x_m=x_1\cdots x_{i-1}\cdot\overline{\overline{x_i\cdots x_j}}\cdot x_{j+1}\cdots x_m.
\end{equation} It is evident that this identity implies the identity~\eqref{equ-epi}.

\smallskip

\emph{Case} 2: $w'$ is a semigroup word. Substitute some letter $x$ to all letters occurring in the identity~\eqref{x_1...x_{m-1}=v'}. Then we obtain an identity $x^{m-1}=x^{m-1+k}$ for some $k>0$. By~\eqref{*x=x^{(p+1)q-1}}, the latter identity implies in the class of all epigroups the identity $\overline x\,=x^{mk-1}$. Using Lemma~\ref{V_Sigma=var_E Sigma} we have
\begin{align*}
\mathcal V'&=\mathbf V[x_1\cdots x_{m-1}=v']=\mathbf V[x_1\cdots x_{m-1}=v',\overline x\,=x^{mk-1}]\\
&=\var_E\{x_1\cdots x_{m-1}=v',\overline x\,=x^{mk-1}\}\ldotp
\end{align*}
As in the Case 1, we apply Theorem~\ref{endomorphisms}. We get that the variety
$$\var_E\{x_1\cdots x_m=x_1\cdots x_{i-1}\cdot\overline{\overline{x_i\cdots x_j}}\cdot x_{j+1}\cdots x_m\}\ldotp$$
contains the variety $\var_E\{x_1\cdots x_{m}=v'x_m,\,\overline x\,x_m=x^{mk-1}x_m\}$.
Note that the variety $\mathbf V[x_1\cdots x_{m}=v'x_m]$ satisfies the identity $x^m=x^{m+k}$. Hence, taking into account~\eqref{*x=x^{pq-1}}, we have that the identity $\overline x\,=x^{mk-1}$ holds in this variety. Then the variety $\mathbf V[x_1\cdots x_{m}=v'x_m]$ satisfies the identity $\overline x\,x_m=x^{mk-1}x_m$. Hence, the identity~\eqref{short equ-epi} holds in the variety $\mathbf V[x_1\cdots x_{m}=v'x_m]$. Then this variety satisfies the identity~\eqref{equ-epi}. Thus, we complete the proof of the implication 1)\,$\longrightarrow$\,4).

\smallskip

2)\,$\longrightarrow$\,1) Let $\mathcal{V\nsupseteq F}$. Then there is an identity $u=v$ that holds in $\mathcal V$ but does not hold in $\mathcal F$. Repeating literally arguments from the proof of the implication 1)\,$\longrightarrow$\,4), we reduce our consideration to the case when the word $u$ is linear. Now Lemma~\ref{fin deg or permut} and the fact that every permutational identity holds in the variety $\mathcal F$ imply that $\mathcal V$ is a variety of finite degree. Theorem~\ref{deg fin epi} is proved.\qed

\medskip

It remains to prove Corollary~\ref{deg n epi}. The implication 1)\,$\longrightarrow$\,4) of this corollary follows from the proof of the same implication in Theorem~\ref{deg fin epi}. The implication 4)\,$\longrightarrow$\,3) is evident, while the implication 3)\,$\longrightarrow$\,2) follows from the evident fact that the variety $\mathcal F_{n+1}$ does not satisfy an identity of the form~\eqref{equ-gen} with $\ell(w)>n$. Finally, the implication 2)\,$\longrightarrow$\,1) of Corollary~\ref{deg n epi} is verified quite analogously to the same implication of Theorem~\ref{deg fin epi}.\qed

\medskip

\begin{flushleft}
Ural Federal University, Institute of Mathematics and Computer Science,\\
Lenina 51, 620000 Ekaterinburg, Russia\\[\smallskipamount]
\emph{E-mail addresses}: sergey.gusb@gmail.com, bvernikov@gmail.com
\end{flushleft}

\end{document}